\documentclass[12pt,reqno]{amsart}
\usepackage{amssymb,amsmath,amscd}
\usepackage{color}
\newtheorem{theorem}{Theorem}

\newtheorem{thm}{Theorem}
\newtheorem{lem}[thm]{Lemma}

\def\Neg{\mathop{\rm col}\nolimits}
\newcommand\DES{{\textsc{DES}}}

\newcommand\remark{\noindent{\bf Remark.}\kern 5pt}
\newcommand\example{\noindent{\bf Example.}\kern 5pt}
\def\si{\sigma }
\def\S{{\mathcal S}}

\newcommand\exc{\mathop{\rm exc}}

\newcommand\Der{{\mathop{\rm Der\kern 0.5pt}}}

\newcommand\ZDer{{\mathop{\rm ZDer\kern 0.5pt}}}
\newcommand\fix{\mathop{\rm fix}}
\newcommand\SIGN{{\textsc{COL}}}
\newcommand\FIX{\mathop{\rm FIX}}

\newcommand\Pos{\mathop{ \rm Pil}}
\newcommand\Pil{\mathop{ \rm Pil}}
\newcommand\Zero{\mathop{ \rm Zero}}
\newcommand\zero{\mathop{ \rm zero}}
\newcommand\rank{\mathop{ \rm red}}

\newcommand\First{ \mathop{\rm L}}
\newcommand\Last{ \mathop{\rm R}}

\def\D{\mathcal{D}}

\def\gln{G_{\ell,n}}
\def\sln{\Sigma_{\ell,n}}

\def\ns{ non-subexcedent }
\def\F{{\bf F}}
\def\f{{\bf\Phi^{-1}}}
\def\ph{{\bf\Phi}}

\def\N{\mathbb{N}}

\def\k+1th{(k+1)^{th}}
\def\dln{\mathcal{D}_{\ell,n}}
\def\dlk{\mathcal{D}_{\ell,k}}

\DeclareMathOperator{\dess}{des}
\DeclareMathOperator{\maj}{maj}
\DeclareMathOperator{\maf}{maf}

\DeclareMathOperator{\fmaj}{fmaj}
\DeclareMathOperator{\fmaf}{fmaf}

\DeclareMathOperator{\Sh}{Sh}

\def\maj{\mathop{\rm maj}\nolimits}
\def\fix{\mathop{\rm fix}\nolimits}
\def\maf{\mathop{\rm maf}\nolimits}

\def\pil{\mathop{\rm pil}\nolimits}
\def\des{\mathop{\rm des}\nolimits}
\def\desl{\mathop{\rm des_\ell}\nolimits}
\def\FIX{\mathop{\rm FIX}\nolimits}

\def\il{\bigl]\kern-.55em\bigl]}
\def\ir{\bigr]\kern-.55em\bigr]}

\def\p-{\Phi^{-1}}
\def\rank{\textrm{rank}}
\numberwithin{equation}{section}
\title{Fix-Euler-Mahonian statistics on wreath products}
\author{Hilarion L. M. Faliharimalala}
\address{D\'epartement de Math\'ematiques et Informatique,
        Universit\'e d'Antananarivo, 101 Antananarivo, Madagascar}
\email{heritianamihanta@yahoo.fr}
\author{Jiang Zeng}
\address{Universit\'e de Lyon, Universit\'e Lyon 1, Institut Camille
Jordan, UMR 5208 du CNred, F-69622, Villeurbanne Cedex, France}
\email{zeng@math.univ-lyon1.fr}
\begin{document}
\begin{abstract}
In 1997 Clarke et al. studied  a $q$-analogue of Euler's difference table
for $n!$ using a key bijection $\Psi$ on symmetric groups.
In this paper we extend their results to  the
wreath product   of a  cyclic group with the symmetric group.
In particular  we obtain a new Mahonian  statistic \emph{fmaf}  on wreath products.
We also show  that Foata and Han's two recent transformations on the symmetric groups
 provide indeed a factorization of $\Psi$.
 \end{abstract}
\date{\today}
\maketitle
\centerline{\emph{Dedicated to Dennis Stanton on the occasion of his 60th birthday}}
\section{Introduction}
For positive integers $\ell, n\geq 1$, define the colored set
$\Sigma_{\ell,n}=\{\zeta^jk \,|\,0\leq j\leq \ell-1, 1\leq k \leq n \}$.
 Let $G_{\ell,n}$  be  the set of  permutations $\sigma$  of 
$\Sigma_{\ell,n}$
   such that  $ \si(\zeta^j\, k)=\zeta^j\si(k)$ for any $k\in [n]$ and $0\leq j\leq \ell-1$.
It is known that
      $G_{\ell,n}$  is isomorphic to the wreath product  $C_\ell \wr S_n$, where $S_n$ denotes 
       the group of permutations of $[n]:=\{1,2,\ldots, n\}$ and $C_\ell$   the $\ell$-cyclic group  generated by $\zeta = e^{2i \pi/ \ell}$.  
In the last two decades many authors
 have being  trying to extend various  enumerative results on symmetric groups
to other classical reflection groups, see
\cite{abr01,ar01, c06, fh08,hlr05,fz07} and the references cited there.
In  particular Adin and Roichman \cite{ar01}
 introduced the flag major index \emph{fmaj} on wreath products of a cyclic group with the symmetric group,
and
Haglund et al.~\cite{hlr05} proved that
\begin{equation}\label{haglund}
\sum_{\sigma \in G_{\ell,n}}q^{\fmaj\sigma}=[\ell]_q[2\ell]_q\cdots[n\ell]_q,
\end{equation}
where $[n]_{q}$ is the $q$-integer $1+q+\cdots +q^{n-1}$.
In this paper we shall consider
 some natural  generalizations of \eqref{haglund} by studying  a wreath product analogue of Euler's $q$-difference table
$\{g_{\ell,n}^m(q)\}_{m\geq n\geq 0}$ defined by the following recurrence:
\begin{align}\label{q-recurrence}
\left \lbrace
\begin{array}{ll}
g_{\ell,n}^n(q)= [\ell]_q[2\ell]_q\cdots[n\ell]_q,\\
g_{\ell,n}^{m}(q)=g_{\ell,n}^{m+1}(q)-q^{\ell(n-m-1)}g_{\ell,n-1}^{m}(q)\qquad
 (0\leq m\leq n-1).
\end{array}
\right.
\end{align}
For example, when
$\ell=1$ and $\ell=2$, the first values of  $g_{\ell,n}^m(q)$ are given as follows:
\begin{itemize}
\item $\ell=1$
{\small
$$\vcenter{
\hbox{$\begin{array}{c|cccccc}
\hbox{$n$}
\backslash\hbox{$m$}&0&1&&2&&3\\
\hline
0&1& & & & &\\
1&0&1&& && \\
2&q&q&&1+q&& \\
3&q[2]_q& q[3]_q&& q^3+2q^2+q&& (1+q)(1+q+q^2)\\
\end{array}$}
}$$}
\item $\ell= 2$
$$\vcenter{
\hbox{$\begin{array}{c|cccccc}
\hbox{$n$}
\backslash\hbox{$m$}&0&1&2&3&&\\
\hline
0&1& & & & &\\
1&q&[2]_q && &&\\
2&q[4]_q+q^2&q[4]_q+q^2[2]_q& [2]_q[4]_q& &&\\
3&g_{2,3}^{0}&g_{2,3}^{1}& g_{2,3}^{2} &[2]_q[4]_q[6]_q&&\\
\end{array}$}
}$$
where
\begin{align*}
g_{2,3}^2&=q^9+3q^7+5q^7+7q^6+8q^5+7q^4+5q^3+3q^2+q,\\
g_{2,3}^1&=q^9+3q^8+5q^7+6q^6+6q^5+5q^4+4q^3+3q^2+q,\\
g_{2,3}^0&=q^9+2q^8+4q^7+4q^6+5q^5+5q^4+4q^3+3q^2+q.
\end{align*}
\end{itemize}
It is remarkable that  $g_{\ell,n}^m(q)$ are polynomials in $q$ with non-negative integral  coefficients.
For $\ell=1$, Clarke et al.~ \cite{chz97}  proved that the entry $g_{1, n}^{m}(q)$ is  actually  the generating function  for a subset
of $S_n$ by  the Mahonian statistic  {\em maf}. Their proof is based on a key bijection $\Psi$
on $S_n$ transforming the statistic $\maf$ to the statistic  {\em major index}.
 Our first aim is  to  show that  the results in  \cite{chz97} can be readily extended
 to wreath products. More precisely, we will  find a combinatorial interpretation and an explicit formula for  $g_{\ell,n}^m(q)$ by  introducing   a new
Mahonian statistic \emph{fmaf} on the wreath products and  by extending  Clarke et al.'s bijection $\Psi$ to the colored permutations.
On the other hand, Foata and Han~\cite{fh08} have recently constructed two new transformations on symmetric groups and
 \emph{noticed}  that
the composition of their two  transformations has  some 
properties in common with  Clarke et al.'s  bijection $\Psi$ in \cite{chz97}. A natural question is then to ask wheather
these two algorithms are \emph{identical}.  Our second aim of this paper is  to settle this open question.

\section{Definitions and main results}
For $x=\zeta^jk\in \Sigma_{\ell,n}$,  we write $x=\varepsilon_x|x|$ with
 $$
  \varepsilon_x= \zeta^j\in C_{\ell}\quad\textrm{ and}\quad |x|=k \in [n].
 $$
If $\ell$ is small,  it is convenient to  write $j$ bars over $i$ instead of $\zeta^j i$, thus
 $\zeta^24=\overline{\overline{4}}$.
By  ordering  the elements of $C_\ell$ as
$1>\zeta>\zeta^2>\cdots >\zeta^{\ell-1}$, we
can define a linear order  on the alphabet
  $\Sigma_{\ell,n}$ as follows:
\begin{equation}\label{ordre}
\varepsilon_{x_1}|x_1|<\varepsilon_{x_2}|x_2| \Leftrightarrow [\varepsilon_{x_1}<\varepsilon_{x_2}] \text{ or } [(\varepsilon_{x_1}=\varepsilon_{x_1}) \text{ and } (|x_1|<|x_2|)].
\end{equation}
For example, the elements of $\Sigma_{4,4}$ are ordered as follows:
$$\bar{\bar{\bar{1}}}<\bar{\bar{\bar{2}}}<\bar{\bar{\bar{3}}}<\bar{\bar{\bar{4}}}<
\bar{\bar{1}}<\bar{\bar{2}}<\bar{\bar{3}}<\bar{\bar{4}}<
\bar{1}<\bar{2}<\bar{3}<\bar{4}<1<2<3<4.
$$
Recall that $i\in [n]$ is a {\it fixed point}
of $\sigma\in G_{\ell,n}$ if $\sigma(i)=i$.
Let $\FIX(\sigma)$  be the set of fixed
 points of $\sigma$ and  $\fix \sigma$ the cardinality of
$\FIX(\sigma)$. The colored permutation $\sigma$ has a \emph{descent} at
 $i\in\{1,2,\ldots, n-1\}$ if $\sigma(i)>\sigma(i+1)$ and
$i$ is called a  \emph{descent place} of $\sigma$.  Let $\DES(\sigma)$ be the set of descent places of $\sigma$.
The \emph{major index}
of $\sigma$, denoted by $\maj \sigma$, is the sum of
all the descent places of $\sigma$.
If $y_1  y_2 \cdots  y_m$ is the word obtained by deleting the fixed points of $\sigma$ and 
$y_i= \varepsilon_i |y_i|$ ($1\leq i\leq m$), writing
$z_i= \varepsilon_i\rank(|y_i|)$ 
 with  `$\rank$"  being
the increasing bijection from $\{|y_1|, |y_2|,\cdots ,|y_m| \}$ to $[m]$,
 then   the \emph{derangement part} of $\sigma$ is defined to be
 \begin{equation}\label{der}
 \Der(\sigma)= z_1z_2\cdots z_m.
 \end{equation}
We now define the main Eulerian and Mahonian statistics on the wreath product $\gln$.
 If $\sigma=x_1\ldots x_n\in \gln$ then the statistics $\desl$, $\exc$ and   $\Neg$  are  defined by
 \begin{align*}
\desl \sigma=\sum_{i=1}^{n-1}\chi(x_i>x_{i+1}),\quad
\exc \sigma=\sum_{i=1}^{n}\chi(x_i>i),\quad
\Neg\sigma=\sum_{j=0}^{\ell-1}j\cdot |\SIGN_j(\sigma)|,
\end{align*}
 where $ \SIGN_j(\sigma)=\{ i\in [n]:\frac{x_i}{|x_i|}=\zeta^j\}$ and $\chi(A)=1$ if $A$ is true and 0 otherwise.
 When $\ell=1$ we shall write $\textrm{des}=\desl$.
 The $\maj$ and $\maf$ statistics are defined by
\begin{align*}
 \maj\sigma=\sum_{i=1}^{n-1}i \cdot \chi(x_i>x_{i+1}),\quad \textrm{and}\quad
  \maf\sigma=\maj\Der(\sigma)+\sum _{j=1}^k(i_j-j),
 \end{align*}
 where $\FIX(\sigma)=\{i_1,\, i_2,\,\ldots, i_k\}$.
The  \emph{flag-maj}  $\fmaj$ and   \emph{flag-maf}  $\fmaf$ statistic   are  defined by
\begin{align*}
\fmaj\sigma=\ell\cdot \maj\sigma + \Neg(\sigma),\quad \textrm{and}\quad
\fmaf\sigma=\ell\cdot \maf\sigma + \Neg(\sigma).
\end{align*}
\remark
While the statistic $\fmaj$ was first introduced by Adin and Roichman~\cite{ar01}, the statistic $\fmaf$ is new and reduces to $\maf$ of Clarke et al.~\cite{chz97} for $\ell=1$.

\example
If $\sigma=\textbf{1}\,\bar{8}\,\textbf{3}\,\textbf{4}\,
6\bar{\bar{2}}\,\textbf{7}\,\bar{5}\,\textbf{9} \in G_{4,9}$,
then  $\maj \sigma=1+5+7=13$, $\FIX\sigma=\{ 1,3,4,7,9 \}$ and
$\Der(\sigma)=\bar{4}\,3\,\bar{\bar{1}}\,\bar{2}$.
 Therefore
 $\maj\Der(\sigma)=2$ and
 $$\maf\sigma=2+\left((1-1)+(3-2)+(4-3)+(7-4)+(9-5)\right)=11.
 $$
 Since
$\SIGN_0(\sigma)=\{1,3,4,5,7,9\},\,
\SIGN_1(\sigma)=\{2,8\},\, \SIGN_2(\sigma)=\{6\},\,\SIGN_3(\sigma)=\emptyset$,
we have  $\Neg(\sigma)=0\times6+1\times2+2\times1=4$.
Finally,
$$
\fmaj\sigma=4\times13+4=56,\quad \textrm{and}\quad
\fmaf\sigma=4\times11+4=48.
$$	

\begin{remark}
Other notions of descents have also been considered previously, see \cite{stein94}.
For the length function of  the wreath product  $C_\ell \wr S_n$  we refer the reader to
\cite{gm06}.
\end{remark}

We first show that the statistics $(\fix, \fmaf)$ and $(\fix,\fmaj)$ are equidistributed on $\gln$ and their common distribution 
has an explicit formula.
\begin{theorem}\label{mainthm1}
The triple statistics $(\fmaf,\exc,\fix)$ and $(\fmaj,\exc,\fix)$ are equidistributed on $\gln$.
Moreover the common generating function
\begin{align}\label{eq:fix-maf}
\sum_{\sigma \in G_{\ell,n}}q^{\fmaf\sigma}x^{\fix\sigma}=\sum_{\sigma \in G_{\ell,n}}q^{\fmaj\sigma}x^{\fix\sigma}
\end{align}
has the explicit formula
\begin{align}\label{eq:xhaglund}
g_{\ell,n}(q,x):=[\ell]_q[2\ell]_q\cdots[n\ell]_q \sum_{k=0}^n\frac{(x-1)(x-q^\ell)\cdots (x-q^{\ell(k-1)})}{[\ell]_q[2\ell]_q\cdots[k\ell]_q}.
\end{align}
\end{theorem}

\begin{remark}
When $\ell=1$
Gessel and Reutenauer~\cite{gr93}  first proved that 
 the generating function  of $(\fmaj, \fix)$ on $S_n$ is given by 
 \eqref{eq:xhaglund}. For general $\ell$ and 
 $x=1$  we recover  Haglund et al's formula~\eqref{haglund}  for the generating function of $\fmaj$ on $\gln$.
 For $x=0$,  we derive from \eqref{eq:fix-maf} and \eqref{eq:xhaglund}  an  explicit formula
for the colored $q$-{\em derangement number}:
\begin{align}\label{(i)}
d_{\ell,n}(q):= \sum_{\sigma\in \D_{\ell,n}}q^{\fmaj\sigma}
 =[\ell]_q[2\ell]_q\cdots[n\ell]_q
 \sum_{k=0}^n\frac{(-1)^kq^{\ell\binom{k}{2}}}{[\ell]_q[2\ell]_q\cdots[k\ell]_q}.
\end{align}
The  $\ell=1$ and $\ell=2$  cases of \eqref{(i)} were first obtained by Gessel  (unpublished) and Wachs~\cite{w89},  and Chow \cite {c06},
 respectively. Finally  \eqref{(i)} yields immediately the following recurrence relation for
 the colored $q$-derangemnt numbers:
\begin{align}\label{(iii)}
d_{\ell,n+1}(q)=
[\ell n+\ell]_qd_{\ell,n+1}(q)+(-1)^{n+1}q^{\ell\binom{n+1}{2}}.
\end{align}
\end{remark}

Introduce the $q$-shifted factorials 
$$
(a;q)_0=1,\quad (a;q)_n=\prod_{k=0}^{n-1}(1-aq^k),\quad n=1,2,\ldots,\,\textrm{or}\,\infty,
$$ 
 then  the $q$-binomial coefficients are defined by
 $$
 {n\brack m}_{q}=\frac{(q;q)_n}{(q;q)_m(q;q)_{n-m}}, \quad  n\geq m\geq 0.
 $$
Instead of the colored $q$-Euler table \eqref{q-recurrence},
 as Clarke et al.~\cite{chz97}, we can consider a more general  triangle than \eqref{q-recurrence} by taking
 $g_{\ell,n}^n(q,x):=g_{\ell,n}(q,x)$ as
  the diagonal coefficients and replace the recurrence relation in \eqref{q-recurrence} by
\begin{align}\label{q-general}
g_{\ell,n}^{m}(q,x)=g_{\ell,n}^{m+1}(q,x)-xq^{\ell(n-m-1)}g_{\ell,n-1}^{m}(q,x)\qquad
 (0\leq m\leq n-1).
\end{align}
For $n\geq m\geq 0$,  denote by  $G_{\ell,n}^m$
the set of permutations $\sigma$ in $G_{\ell,n}$ such that
  $\FIX(\sigma)\subset \{n-m+1, \ldots, n-1, n\}$, i.e., $\sigma(i)\neq i$ for $i\leq n-m$.
  In particular we have $G_{\ell,n}^n=G_{\ell,n}$ and
 $\dln:=G_{\ell,n}^0$  is the set of colored derangements  of order $n$.
The following theorem  gives a full description of $g_{\ell,n}^{m}(q,x)$, which  generalizes  the previous results  in \cite{chz97} for $\ell=1$
and \cite{fz07} for  $q=1$, respectively.
\begin{theorem}\label{mainthm2}
 For $n\geq m\geq  0$ we have
 the following explicit formula:
\begin{align}\label{explicitgeneral}
g_{\ell,n}^{m}(q,x)=\sum_{k=0}^{n-m}(-x)^{k}{n-m\brack
 k}_{q^\ell}q^{\ell{k\choose 2}}g_{\ell,n-k}(q,x)
\end{align}
and the combinatorial interpretation:
\begin{equation}\label{eq:eqcombi2}
 g_{\ell,n}^{m}(q,x)
=\sum_{\sigma \in G_{\ell,n}^m}q^{\fmaf\sigma}x^{\fix\sigma}.
\end{equation}
\end{theorem}
\begin{remark}  The two statistics \emph{fmaf} and  \emph{fmaj}  are identical on  $\dln$, equidistributed on
$G_{\ell, n}$,  but not
equidistributed  on the set   $G_{\ell,n}^m$ for  $0 < m < n$.
\end{remark}

Let $0\leq m\leq n$ and $v$ be  a nonempty word of length $m$ on the alphabet $[n]$.
 Denote by
$\Sh(0^{n-m}v)$ the set of all
\emph{shuffles} of the words
$0^{n-m}$ and $v$, that is, the set of all words it is possible to construct
using $(n-m)$ 0's and  the  letters in $v$ by preserving the order of all the letters in $v$.
For any word  $w=x_1\ldots x_n$ in $\Sh(0^{n-m}v)$,
we call \emph{pillar}  any non zero  letter of
 $w$ and write
$\Pos(w):=v$ and $\Zero w:=\{i: 1\leq i\leq n,x_i=0\}$.
Let $\pil w$ be the length of $\Pos w$ and $\zero w$ be the cardinality of $\Zero w$.
Clearly $w$ is completely characterized by the pair $(\Zero w, \Pos w)$.
Given a  permutation $\sigma=\sigma(1)\sigma(2)\ldots \sigma(n+m)\in {S}_{n+m}$,
let $(j_1,j_2,\ldots, j_m)$ be the increasing sequence of the integers $k$ such that
$\sigma(k)\neq k$ for $1\leq k\leq n$. The word $w=x_1x_2\ldots x_{n+m}$ derived from $\sigma$ by replacing each fixed point
by 0 and each other letter $\sigma(j_k)$
by $\rank(\sigma(j_k))$ will be denoted by $\ZDer(\sigma)$.
For example, if
 $\sigma=\textbf{1}\,8\,\textbf{3}\,\textbf{4}\,6 2\,
\textbf{7}\,5 \,\textbf{9} \in S_9$  then
 $\ZDer(\sigma)=0\,4\,0\,0\,3 1\,0\,2 \,0$.
Let
$${\S}_n^{\Der}:=\bigcup\limits_{m,v}
\Sh(0^{n-m}v) \qquad (0\le m\le n,\, v\in \D_m).
$$
It is obvious (see \cite[Proposition 1.3]{fh08})  that
the map $\ZDer$ is a bijection from
${S}_n$ to ${\S}_n^{\Der}$. Recently 
Foata and Han  \cite{fh08} have constructed two transformations
$\mathbf \Phi$ and $\F:=\F_3$ on ${\S}_n^{\Der}$ (see Section~\ref{preuveegal}), apparently related to $\Psi$.
Our third object is to show  that  their  two bijections provide indeed a factorization of 
   Clarke et al.'s  bijection $\Psi$.
 \begin{theorem}\label{mainthm3}
We have
\begin{align}\label{eq:FHdecom}
\Psi=\ZDer^{-1}\circ {\bf F}\circ  {\bf\Phi}^{-1}\circ \ZDer.
\end{align}
In other words,
the diagram in
Figure~1 is commutative.
 \end{theorem}

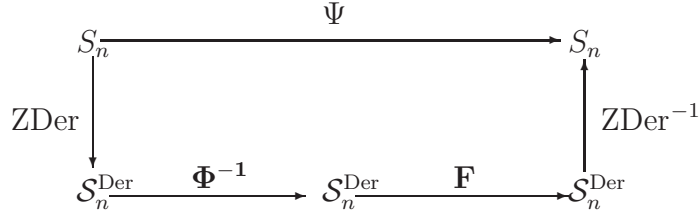
\begin{figure}
\setlength{\unitlength}{0.62pt}
\begin{picture}(520,120)
\put(100,7){$\S_n^{\Der}$}
\put(100,98){$S_n$}
\put(400,98){$S_n$}
\put(250,7){$\S_n^{\Der}$}
\put(400,7){$\S_n^{\Der}$}
\put(115, 105){\vector(1,0){280}}
\put(110,95){\vector(0,-1){68}}
\put(410,25){\vector(0,1){68}}
\put(120,10){\vector(1,0){120}}
\put(270,10){\vector(1,0){130}}
\put(60,50){$ \ZDer$}
\put(420,50){$\ZDer^{-1}$}
\put(330,15){$\F$}
\put(170,15){$\bf \Phi^{-1}$}
\put(250,115){$\Psi$}
\end{picture}
\medskip
\caption {The Foata-Han factorization of  Clarke et al.'s  bijection $\Psi$}
\end{figure}

\medskip

The rest of this paper is organized as follows.
In Section~\ref{proof of th mainthm1}, after  recalling  Clarke et al.'s  bijection
$\Psi$,  we prove Theorem~\ref{mainthm1}.
The proof of  Theorem~\ref{mainthm2} is given 
 in Section~\ref{forme explicite et interpretation combinatoir}. 
 In Section~\ref{preuveegal},  after recalling Foata and Han's  two transformations $\mathbf \Phi$ and $\F$,
we prove  Theorem~\ref{mainthm3} and postpone  the proof of  three technical  lemmas 
to  Section~\ref{The proof of three lemmas}.

\section{Proof of Theorem~\ref{mainthm1}}\label{proof of th mainthm1}
We first recall Clarke et al.'s bijection $\Psi$, which will also be used in Section~\ref{preuveegal},  and then 
show that one can
extend  $\Psi$ to  $G_{\ell,n}$  with the following property:
\begin{align}\label{eq:mainthm1}
(\maf,\exc,\fix, \Neg)\Psi(\sigma)=(\maj,\exc,\fix, \Neg) \sigma,\qquad \forall \sigma\in \gln.
\end{align}

\subsection{Clarke et al.'s bijection $\Psi$}\label{Psi}\
Let $\sigma=x_1x_2\ldots x_n\in S_n$ with
 $x_0=x_{n+1}=+\infty$.
 For $0\leq i\leq n$, a pair $(i,i+1)$ of positions
 is the $j$-th {\em slot} of $\sigma$ \emph{provided that} $x_i\ne i$ and that $\sigma$ has $i-j$ fixed points
 $f$ such that $f<i$. Of course,  the $j$-th slot  is $(j,j+1)$ if $\sigma$ is a derangement.
 Clearly we can insert a fixed point into the $j$-th
 slot and  obtain the permutation
\begin{equation}\label{jslot}
\langle \sigma,j\rangle=x_1'x_2'\ldots x_i'\ (i+1)\ x_{i+1}'\ldots x_n',
\end{equation}
where $x'=x$ if $x\le i$ and $x'=x+1$ if $x>i$ for all $ x \in [n]$.

Now,   if $\sigma$ is a derangement in $S_n$ and $(i_1,
i_2,\ldots, i_m)$ a sequence of integers such that $0\le i_1\le
i_2\le\cdots\le i_m\le n$,
we can insert successively $m$ fixed points in $\sigma$ and obtain   a permutation $\tau$ in $S_{n+m}$:
\begin{align}\label{eq:writing}
\tau=\langle\sigma,i_1,\ldots,i_m\rangle=\langle\langle\sigma,i_1,\ldots,i_{m-1}\rangle,i_m\rangle.
\end{align}
Note that  $\Der(\tau)=\sigma$ and
the fixed points of the last permutation are
$i_1+1,i_2+2,\ldots,i_m+m$.
Conversely, any permutation $\tau\in S_{m+n}$ with   $m$ fixed points
can be written  as  \eqref{eq:writing} in one and only one way.
Thus,  if  $S(\sigma,m)$ denotes the set of permutations in $S_{n+m}$ with derangement part $\sigma\in
\D_n$, then
$$
S(\sigma,m)=\{\langle\sigma, i_1,\ldots, i_m\rangle|\; 0\leq i_1\leq i_2\leq \cdots \leq i_m\leq n \}.
$$
Let $\sigma\in S_n$.  The $j$-th slot $(i,\,
i+1)$ of $\sigma$ is said to be  \emph{green} if
$\dess\langle\sigma,j\rangle=\dess\sigma$, \emph{red} if
 $\dess\langle\sigma,j\rangle=\dess\sigma+1$.
We assign \emph{values} from 0 to $g$ to the green slots  of  $\sigma$ from right to
 left, and values  from $g+1$ to
 $n$ to the  red slots from left to right. Denote the value of the $j$-th slot by
$g_j$.
The  bijection $\Psi: S(\sigma,m)\longrightarrow S(\sigma,m)$ is defined by induction on $m\geq 0$ as follows:
\begin{enumerate}
\item $\Psi$ is the identity mapping on $S(\sigma,0)$.
\item
Define $\Psi$ on $S(\sigma, 1)$ by
\begin{equation}\label{m=1}
\Psi\langle\sigma,i\rangle=\langle\sigma,g_i\rangle.
\end{equation}
\item Let $m>1$ and suppose that
 $\Psi$ has been defined on $S(\sigma, k)$ for $0\leq k\leq m-1$.
Consider $\tau=\langle\sigma,i_1,\ldots,i_m\rangle$.
Suppose that the $i_m$-th slot of $\sigma$
is green.  Then, if
$\Psi\langle\sigma,i_1,\ldots,i_{m-1}\rangle=\langle\sigma,j_2,\ldots,j_m\rangle$,
we define
\begin{align}\label{eq:green}
\Psi(\tau)=\langle\sigma,g_{i_m},j_2,\ldots,j_m\rangle.
\end{align}
Suppose that the $i_m$-th slot of $\sigma$ is red.
Let $k$ be the smallest positive integer such
that $i_{m-k}<i_m$. Thus $i_{m-k}<i_{m-k+1}=\cdots=i_m$. Then,
\begin{itemize}
  \item if $k=m$ we define
  \begin{align}\label{eq:red1}
\Psi(\tau)=\langle\sigma,
\underbrace{g_{i_m}-i_m,\ldots,g_{i_m}-i_m}_{m-1\mbox{ terms}},g_{i_m}\rangle.
\end{align}
 \item if $k<m$ and $\Psi\langle\sigma,i_1,\ldots,i_{m-k}\rangle=\langle\sigma,j_1,\ldots,j_{m-k}\rangle$,
we define
\begin{align}\label{eq:red2}
\Psi(\tau)=\langle\sigma,
\underbrace{g_{i_m}-i_m,\ldots,g_{i_m}-i_m}_{k-1\mbox{ terms}},
j_1+1,\ldots,j_{m-k}+1,g_{i_m}\rangle.
\end{align}
\end{itemize}
\end{enumerate}

\subsection{Generalization of $\Psi$ to $G_{\ell,n}$}
To extend  the insertion algorithm  \eqref{eq:writing} to colored permutations we just  need to modify  equation \eqref{jslot} as follows:
    \begin{equation}\label{jslotG}
\langle \sigma,j\rangle=x_1'x_2'\ldots x_i'\ (i+1)\ x_{i+1}'\ldots x_n',
\end{equation}
where $x'=x$ if $|x|\le i$ and
\begin{equation}\label{somme}
x'= x+1:=\varepsilon_x(|x|+1),
\end{equation}
 if $|x|>i$ for all $ x \in \sln$.
Thus each colored permutation $\tau\in \gln$ can be written as \eqref{eq:writing}.

\medskip
\example Let $\tau=1\,\,\bar 9\,\,3\,\,10\,\,5\,\,6\,\,7\,\,4\,\,2\,\,\bar{\bar8}\in G_{3,10}$.  Then $\sigma:=\Der(\tau)=\bar{4}\,5\,2\,1\,\bar{\bar{3}}\in  {\mathcal D}_{3,5}$ and
$\tau=\langle \sigma, 0,1,2,2,2\rangle$. Indeed,
\begin{align*}
\langle \sigma, 0\rangle&=1\,\,\bar 5\,\,6\,\,3\,\,2\,\,\bar{\bar 4},\\
\langle \sigma, 0,1\rangle&=1\,\,\bar 6\,\,3\,\,7\,\,4\,\,2\,\,\bar{\bar 5},\\
\langle \sigma, 0,1,2\rangle&=1\,\,\bar 7\,\,3\,\,8\,\,5\,\,4\,\,2\,\,\bar{\bar 6},\\
\langle \sigma, 0,1,2,2\rangle&=1\,\,\bar 8\,\,3\,\,9\,\,5\,\,6\,\,4\,\,2\,\,\bar{\bar7},\\
\langle \sigma, 0,1,2,2,2\rangle&=1\,\,\bar 9\,\,3\,\,10\,\,5\,\,6\,\,7\,\,4\,\,2\,\,\bar{\bar8}.
\end{align*}
So
$\tau=\langle \sigma, 0,1,2,2,2\rangle$.

\medskip
As in the symmetric group we say that the $i$-th
slot ($0\leq i\leq n$)  of $\sigma\in \gln$ is green if $\desl\langle \sigma,i\rangle=\desl\sigma$ and red
 if $\desl\langle \sigma,i\rangle=\desl\sigma+1$.
 In the same way, we can assign  a value to each
 slot of $\sigma$. Now, let $\sigma$ be a derangement
and  $G(\sigma, m):=\{\tau\in G_{\ell,n+m}\,|\,\Der(\tau)=\sigma\}$.
We can  extend  $\Psi$ to a bijection on $G(\sigma, m)$.
It follows  that
$\Der(\Psi(\tau))=\Der(\tau)$ for any $\tau\in G(\sigma, m)$. Since
$\exc\tau=\exc(\Der(\tau))$ and $\Neg\tau=\Neg(\Der(\tau))$,  we have immediately
$$
(\exc,\fix, \Neg)\Psi(\tau)=(\exc,\fix, \Neg) \tau.
$$
It remains to verify $\maj \tau=\maf(\Psi(\tau))$.
On symmetric groups, the proof of the latter equality is based on the following result.
\begin{lem}[Clarke et al.]\label{lem clark}
Let $\sigma$ be a derangement in symmetric group and  $g_i$ the value of its $i$-th  slot then
\begin{equation}\label{maj-maj}
\maj\langle \sigma,i\rangle=\maj\sigma+g_i.
\end{equation}
\end{lem}
Now, the substitution $x\rightarrow x'$ in  \eqref{jslotG} is compatible with
 the linear order (see \eqref{ordre}) on the alphabet $\sln$, namely
\[   \forall a,b \in \sln\quad a<b \Longleftrightarrow a'<b',\]
hence
 $\DES(x_1'x_2'\ldots x_i')=\DES(x_1x_2\ldots x_i)$ and $\DES( x_{i+1}'\ldots x_n')=\DES( x_{i+1}\ldots x_n)$.
So  the  proof of  Lemma~\ref{lem clark} in  \cite{chz97} remains valid when we replace a derangement $\sigma$
 by a any colored derangement.  Then we derive  \eqref{eq:mainthm1} as in \cite{chz97}.

\medskip
 \example  Let $\tau= 1\,\,\bar 9\,\,3\,\,10\,\,5\,\,6\,\,7\,\,4\,\,2\,\,\bar{\bar8}\in G_{3,10}$.
 We have $\DES\tau=\{1,4,7,8,9\}$ and $\maj\tau=29$, $\Der(\tau)=\sigma=  \overline{4} \,5\,2\,\, 1 \,\overline{\overline{3}}$; $\maj\sigma=9$.
 By the previous example, we can write $\tau=\langle  \sigma \,,0,1,2,2,2   \rangle$. Note that slots 1,3,4,5 are {\em green}, while slots 0 and 2 are {\em red}.
 So the sequence of values of the slots is  $(g_0,g_1,g_2,g_3,g_4,g_5)=(4,3,5,2,1,0)$.
 The algorithm $\Psi$ goes as follows:
 \begin{itemize}
\item $\Psi\langle \sigma,0 \rangle=\langle \sigma$,  $g_0 \rangle=\langle \sigma,4 \rangle$;
\item Since 1 is green, $\Psi\langle \sigma, 0, 1 \rangle=\langle \sigma, g_1,4\rangle=\langle \sigma, 3,4\rangle$;
\item Since  2 is red,  $\Psi\langle \sigma, 0, 1,2,2,2 \rangle=\langle \sigma, g_2-2,g_2-2,3+1,4+1,g_2\rangle=\langle \sigma, 3,3,4,5,5\rangle$.
\end{itemize}
Thus $\Psi(\tau)={\bar6}\,\,8\,\,2\,\,4\,\,5\,\,1\,\,7\,\,\bar{\bar3}\,\,9\,\,10$. Note that
\begin{align*}
 \maf\Psi(\tau)&=\maj\sigma+3+3+4+5+5=29=\maj(\tau),\\
 \fmaf\Psi(\tau)&=\fmaj(\tau)=3\times 29+3=90.
 \end{align*}
\subsection{Proof of   \eqref{eq:xhaglund}}
Recall  Cauchy's $q$-binomial formula  (cf. \cite[p.17]{andrews76}):
$$
\sum_{n\geq 0}\frac{(a;q)_n}{(q;q)_n}u^n=\frac{(au;q)_\infty}{(u;q)_\infty}.
$$
In particular we have  Euler's formula by taking $a=0$:
\begin{align}\label{eq:euler}
e_q(u):=\sum_{n\geqslant 0}\frac{u^n}{(q;q)_n}=\frac{1}{(u;q)_\infty}.
\end{align}
Let $f_{\ell,n}(q,x)=\sum_{\sigma\in \gln}x^{\fix\sigma}q^{\fmaf\sigma}$.
By the insertion algorithm,
we can write any permutation $\sigma\in G_{\ell,n}$  with $n-k$ ($0\leq k\leq n$) fixed points as
$\sigma=\langle \pi, i_1, \ldots, i_{n-k}\rangle$, where $\pi\in \D_{\ell, k}$ and $0\leq i_1\leq \cdots \leq i_{n-k}\leq k$.  Since
$\fmaf \sigma=\fmaj\pi+\ell(i_1+\cdots +i_{n-k})$, we have 
\begin{align*}
f_{\ell,n}(q,x)&=\sum_{k=0}^nx^{n-k}\sum_{\pi\in \dlk}\sum_{\sigma\in \gln\atop \Der(\sigma)=\pi}q^{\fmaf\sigma}\\
&=\sum_{k=0}^nx^{n-k}\sum_{\pi\in \dlk}q^{\fmaj\pi}
\sum_{0\leq i_1\leq\cdots\leq i_{n-k}\leq k}q^{\ell(i_1+\cdots +i_{n-k})}\\
&= \sum_{k=0}^nx^{n-k}{n\brack k}_{q^{\ell}}d_{\ell,k}(q).
\end{align*}
Therefore
\begin{align}\label{eq:keyid}
 \sum_{n\geqslant 0}f_{\ell,n}(q,x)\frac{u^n}{(q^\ell;q^\ell)_n}=
 e_{q^\ell} (xu)\sum_{k=0}^\infty d_{\ell,k}(q)\frac{u^k}{(q^\ell;q^\ell)_k}.
\end{align}
Since 
 $f_{\ell,n}(q,1)=(1-q)^{-n}(q^\ell;q^\ell)_n$ by \eqref{haglund} and \eqref{eq:fix-maf},  setting $x=1$ in \eqref{eq:keyid} yields then
\begin{align*}
 \sum_{n\geq 0}d_{\ell,n}(q)\frac{u^n}{(q^\ell;q^\ell)_n}=\frac{1}{1-u/(1-q)}\frac{1}{e_{q^\ell} (u)}.
 \end{align*}
Substituting this back  to \eqref{eq:keyid} we obtain
\begin{align}\label{eq:f}
\sum_{n\geqslant 0}f_{\ell,n}(q,x)\frac{u^n}{(q^\ell;q^\ell)_n}=\frac{1}{1-u/(1-q)}\frac{e_{q^\ell}(xu)}{e_{q^\ell} (u)}.
\end{align}
On the other hand, by \eqref{eq:xhaglund} we have
\begin{align*}
\sum_{n\geq 0}g_{\ell,n}(q,x) \frac{u^n}{(q^\ell,q^\ell)_n}&=
 \sum_{n\geq 0}\sum_{i\geq 0}
 \frac{(x-1)(x-q^{\ell})\cdots (x-q^{\ell(i-1)})}
 {(q^\ell;q^\ell)_i}\frac{u^n}{(1-q)^{n-i}}\\
&=\sum_{i\geq 0}\frac{(x^{-1};q^\ell)_i}{(q^\ell;q^\ell)_i}(xu)^i \sum_{n\geq 0}\left(\frac{u}{1-q}\right)^n\\
&=\frac{1}{1-u/(1-q)}\frac{(u;q^\ell)_\infty}{(xu;q^\ell)_\infty},
\end{align*}
which is equal to the right-hand side of
\eqref{eq:f}  by Euler's formula \eqref{eq:euler}. It follows that $f_{\ell,n}(q,x)=g_{\ell,n}(q,x)$.

\begin{remark}
As $(1-u)e_q(u)=e_q(qu)$, we can also write \eqref{eq:f} as
$$
 \sum_{n\geqslant 0}f_{\ell,n}(q,x)\frac{u^n}{(q^\ell;q^\ell)_n}=
\frac{(1-q)e_{q^\ell}(xu)}{e_{q^\ell}(q^\ell u)-qe_{q^\ell}(u)}.
$$
\end{remark}

\section{Proof of Theorem~\ref{mainthm2}}\label{forme explicite et interpretation combinatoir}
To derive an explicit formula for $g_{\ell,n}^m(q)$ we give
a more general formula, which is a variant of a result in \cite[Th.
 3]{z06} and may be also interesting in its own right.
\begin{lem}
Let $(a_{n,m})_{0\leq m\leq n}$ be an array defined by
\begin{align}\label{eq:array}
\left \lbrace
\begin{array}{ll}
a_{0,m}=x_m\qquad \qquad (m=n);\\
a_{n,m}=z_ma_{n-1,m+1}+y_na_{n-1,m}\qquad (0\leq m\leq n-1).
\end{array}
\right.
\end{align}
If $e_i(y_1,y_2,\ldots,y_n)$  denotes the
$i$-th elementary symmetric polynomial of
$y_1,\ldots,y_n$, then
\begin{align}\label{mainresul}
a_{n,m}=\sum_{k=0}^nx_{m+k}(z_mz_{m+1}\ldots z_{m+k-1})
e_{n-k}(y_1,y_2,\ldots, y_n).
\end{align}
\end{lem}
\begin{proof} The formula is obviously true for $n=0$ and $n=1$.
Suppose that it is true until $n-1$.
Since
$(1+y_1t)(1+y_2t)\cdots (1+y_nt)=\sum_{i=0}^ne_i(y_1,\ldots, y_n)t^i$,
we then have
\begin{align*}
a_{n,m}&=y_na_{n-1,m}+z_ma_{n-1,m+1}\\
&=y_n\sum_{k=0}^{n-1}(z_mz_{m+1}\ldots z_{m+k-1})
e_{n-1-k}(y_1,y_2,\ldots, y_{n-1})x_{m+k}\\
&\hspace{1cm}+z_m
\sum_{k=0}^{n-1}(z_{m+1}z_{m+2}\ldots z_{m+k})e_{n-1-k}(y_1,y_2,\ldots,
 y_{n-1})x_{m+k+1}\\
&=y_ne_{n-1}(y_1,y_2,\ldots, y_{n-1})x_m\\
&+\sum_{k=0}^{n-1}\left(\prod_{j=0}^{k}z_{m+j}\right)
\left(y_ne_{n-2-k}(y_1,y_2,\ldots, y_{n-1})
+e_{n-1-k}(y_1,y_2,\ldots, y_{n-1})\right)x_{m+1+k}\\
&=e_{n}(y_1,y_2,\ldots, y_{n})x_m+
\sum_{k=1}^{n}\left(\prod_{j=0}^{k-1}z_{m+j}\right)
e_{n-k}(y_1,y_2,\ldots, y_{n})x_{m+k}.
\end{align*}
This completes the proof.
\end{proof}

Now, specializing the array \eqref{eq:array} with
  $x_m=g_{\ell,m}^m(q,x)$, $z_m=1$  and $y_n=-xq^{\ell (n-1)}$, then
the  $q$-binomial formula
$(1+t)(1+qt)\cdots (1+q^{n-1}t)
=\sum_{k=0}^n {n\brack k}_q
q^{k\choose 2}t^k$
implies that
\begin{align*}
e_k(1,q,q^2,\ldots, q^{n-1})={n\brack k}_q
q^{k\choose 2}\qquad (0\leq k\leq n).
\end{align*}
Applying \eqref{mainresul} we get
\begin{align*}
a_{n,m}=g_{\ell,n+m}^{m}(q,x)&=\sum_{k=0}^n(-x)^{n-k}{n\brack
 k}_{q^\ell}q^{\ell{n-k\choose 2}}
g_{\ell,m+k}^{m+k}(q,x).
\end{align*}
Shifting $n$ by $n-m$ and then  replacing $k$ by $n-m-k$ yields \eqref{explicitgeneral}.
\medskip

Let $f_{\ell,n}^{m}(q,x):= \sum q^{\fmaf\sigma} x^{\fix\sigma}$ ($\sigma \in G_{\ell,n}^m$)  be the right-hand side of  \eqref{eq:eqcombi2}.
Then $f_{\ell,n}^{n}(q,x)=g_{\ell,n}^{n}(q,x)$ by Theorem~\ref{mainthm1}.
For each fixed $n$
we will show by induction on $m$ ($0\leq m\leq n$) that
  $\{f_{\ell,n}^{m}(q,x)\}$  satisfies the recurrence relation \eqref{q-recurrence}.
For  $0\leq m\leq n-1$ define
 $$
E:=\gln^{m+1}\setminus \gln^{m}=\{ \sigma \in\gln^m : \sigma(n-m)=n-m\}.
 $$
 By \eqref{eq:writing}, each permutation $\sigma\in E$ can be written as
  $$
 \sigma=\langle\Der \sigma, i_1-1,i_2-2,\ldots i_r-r\rangle,
 $$
 where $i_1, \ldots, i_{r}$ are the  fixed points of  $\sigma$ arranged
in increasing order ($i_1=n-m$).
 Let  $\sigma'=\langle\Der \sigma, i_2-2,\ldots i_r-r\rangle$. Then the mapping
$\sigma\mapsto \sigma'$ is a bijection
from  $E$ to  $G_{\ell,n-1}^m $ such that $\fix \sigma=\fix\sigma'+1$ and
$ \fmaf \sigma=\ell((i_1-1)+(i_2-2)+\cdots+(i_r-r))+\fmaj\Der\sigma$.
It follows that   $\fmaf\sigma=\fmaf\sigma'+\ell(n-m-1)$.  Hence
\begin{align*}
 f_{\ell,n}^{m+1}(q,x)
 &=\sum_{\sigma \in G_{\ell,n}^m}q^{\fmaf\sigma}x^{\fix\sigma}+\sum_{\sigma' \in
 G_{\ell,n-1}^m}q^{\fmaf\sigma'+\ell(n-m-1)}x^{\fix\sigma}\\
 &=f_{\ell,n}^{m}(q,x)+xq^{\ell(n-m-1)}f_{\ell,n-1}^{m}(q,x).
\end{align*}
This completes the proof of \eqref{eq:eqcombi2}.

\section {Proof of Theorem \ref{mainthm3}}\label{preuveegal}
\subsection{Foata-Han's first  transformation $\ph$}
Let $v$ be a  derangement of order~$m$ and
$w=x_1x_2\cdots x_n \in \Sh(0^{n-m}v)$ $(0\le n\le m)$. Thus
$v=x_{j_1}x_{j_2}\cdots x_{j_m}$, where
$1\le j_1<j_2<\cdots <j_m\le n$.  Recall that ``rank" is the
increasing bijection of $\{j_1,j_2,\ldots,j_m\}$ onto the interval
$[\,m\,]$. A positive letter~$x_k$ of~$w$ is said to
be excedent (resp. subexcedent) if 
$x_k>\rank(k)$ (resp. $x_k<\rank(k)$). Accordingly, a letter is
non-subexcedent if it is either equal to~0 or excedent.

We define $n$ bijections~$\phi_l$ ($1\leq l\leq n$)
from $\Sh(0^{n-m}v)$ ($0\leq n\leq m$) onto itself in the following manner:
for each~$l$ such that $n-m+1\le l\le n$ let
$\phi_l(w):=w$. When $1\le l\le n-m$, let $x_j$ denote the $l$-th
letter of~$w$, equal to~0, when~$w$ is read {\it from left to right}.
Three cases are next considered (by convention,
$x_0=x_{n+1}=+\infty$):
\begin{itemize}
\item[(1)] $x_{j-1}$, $x_{j+1}$ both non-subexcedent;
\item[(2)] $x_{j-1}$ non-subexcedent, $x_{j+1}$ subexcedent; or
$x_{j-1}$, $x_{j+1}$ both subexcedent with $x_{j-1}>x_{j+1}$;
\item[(3)] $x_{j-1}$ subexcedent, $x_{j+1}$ non-subexcedent; or
$x_{j-1}$, $x_{j+1}$ both subexcedent with $x_{j-1}<x_{j+1}$.
\end{itemize}
When case (1) holds, let $\phi_l(w):=w$.
When case (2) holds, determine the {\it greatest} integer
$k\ge j+1$ such that
$x_{j+1}<x_{j+2}<\cdots <x_k< \rank(k)$
so that
$w=x_1\cdots x_{j-1}\;0\; x_{j+1}\cdots
x_k\;x_{k+1}\cdots x_n$
and define:$
\phi_l(w):=x_1\cdots x_{j-1}\;
x_{j+1}\cdots x_k\;0\;x_{k+1}\cdots x_n$.
When case (3) holds, determine the {\it smallest}
integer $i\le j-1$ such that
$\rank(i)>x_i>x_{i+1}>\cdots >x_{j-1}$ so that
$w=x_1\cdots x_{i-1}\ x_i\cdots x_{j-1}\;0\;
x_{j+1}\cdots x_n$
and define
$\phi_l(w):=x_1\cdots x_{i-1}\; 0\;
x_{i}\cdots x_{j-1}\;x_{j+1}\cdots x_n.$

It is important to note that $\phi_l$ has no action on the 0's
other than the $l$-th one. The mapping ${\bf\Phi}$  is defined to be the composition product
${\bf\Phi}:=\phi_1\phi_2\cdots \phi_{n-1}\phi_n$.
To verify that $\Phi$ is bijective, Foata and Han introduce a class of bijections
$\psi_l$, whose definitions are parallel to the definitions of the $\phi_l's$.
Let
$w=x_1x_2\cdots x_n\in \Sh(0^{n-m}v)$ $(0\le n\le m)$.
We define $n$ bijections $\psi_l$
$(1\leq l\leq n)$ of~$\Sh(0^{n-m}v)$ onto itself in the following
manner: For each $l$ such that $n - m + 1 \leq l\leq n$ let $\psi_l(w) = w$.
When $1 \leq l \leq n-m$,
let $x_j$ denote the $l$-th letter of $w$, equal to 0, when $w$
is read from left to right.  Consider the following three cases (remember
$x_0 = x_{n+1} = +\infty$ by convention):
\begin{enumerate}
  \item[(1')]  $x_{j-1}$, $x_{j+1}$ both non-subexcedent;
  \item [(2')]$x_{j-1}$  subexcedent,  $x_{j+1}$\ns  or  $x_{j-1}$,
  $x_{j+1}$ both  subexcedent with $x_{j-1}>x_{j+1}$;
  \item[(3')] $x_{j-1}$ \ns , $x_{j+1}$   subexcedent  or   $x_{j-1}$,
  $x_{j+1}$  both   subexcedent with $x_{j-1}<x_{j+1}$.
\end{enumerate}
When case (1') holds, let $\psi_l(w) := w$.
When case (2') holds, determine the smallest integer k $\leq j -1$ such that
$ x_k  < \cdots <x_{j-1}<\rank(j-1) $
so that $w = x_1 \cdots x_{k-1} x_k \cdots x_{j-1}  0  x_{j+1}  \cdots x_n$
and define: $\psi_l(w) := x_1 \cdots x_{i-1}  0  x_{i} \cdots x_{j-1} x_{j+1} \cdots x_n.$
When case (3') holds, determine the greatest integer $k \geq j+1$ such that
$\rank(j+1)> x_{j+1} > x_{j+2} > \cdots > x_k $ so that
$w = x_1 \cdots x_{j-1} 0 x_{j+1} \cdots x_k x_{k+1} \cdots x_n$
and define:$
\psi_l(w) := x_1 \cdots x_{j-1} x_{j+1} \cdots x_k  0  x_{k+1} \cdots x_n.$
As shown in \cite{fh08}, the product  ${\bf\Phi}^{-1}:=\psi_n\psi_{n-1}\cdots  \psi_1$ is the
inverse bijection of $\bf \Phi$.
\subsection{Foata-Han's second transformation $\F$}
The bijection ${\bf F}$  maps each shuffle
class $\Sh(0^{n-m}v)$ with $v$ an arbitrary word of length~$m$
$(0\le m\le n)$ onto itself. When $n=1$ the unique element of the
shuffle class is sent onto itself.
Also let ${\bf F}(w)=w$ when $\des(w)=0$.
Let $n\ge 2$ and assume that~${\bf
F}(w')$ has been defined for all words $w'$ with nonnegative
letters, of length $n'\le n-1$. Let
$w$ be a word of length~$n$ such that $\des(w)\geq 1$. We may write
$$
w=w'a0^rb,
$$
where $a\ge 1$, $b\ge 0$ and $r\ge 0$. Three cases are
considered:
$$
(1) \quad a\le b;\quad (2) \quad a>b, \quad r\ge 1;\quad (3) \quad a>b, \quad r=0 .
$$
In case~(1) define ${\bf F}(w)={\bf F}(w'a0^rb):=
({\bf F}(w'a0^r))b$. In case~(2) we may write ${\bf F} (w'a0^r)=w''0$ and then define
\begin{align*}
\gamma\,{\bf F}(w'a0^r)&:=0w'';\\
{\bf F}(w)={\bf F}(w'a0^rb)&:=(\gamma\,{\bf F}(w'a0^r))b=0w''b.
\end{align*}
In short, add one letter~``0" to the left of ${\bf F}(w'a0^r)$, then
delete the rightmost  letter~``0"  and add $b$ to the right.
In case (3) remember that $r=0$. Write
$$
{\bf F}(w'a)=0^{m_1}x_1v_10^{m_2}x_2v_2\cdots 0^{m_k}x_kv_k,
$$
where $m_1\ge 0$, $m_2,\ldots,m_{k}$ are
all positive, then
$x_1$, $x_2$, \dots~, $x_k$ are positive letters and $v_1$,
$v_2$, \dots~, $v_k$ are words with positive letters, possibly
empty. Then define:
\begin{align*}
\delta\,{\bf F}(w'a)&:=
x_10^{m_1}v_1x_20^{m_2}v_2x_3\cdots x_k0^{m_k}v_k;\\
{\bf F}(w)={\bf F}(w'ab)&:=(\delta\,{\bf F}(w'a))b.
\end{align*}
In short, move each positive letter occurring just after a 0-factor
of ${\bf F}(w'a)$ to the beginning of that 0-factor and add~$b$
to the right.

\subsection{Proof of Theorem \ref{mainthm3}}
We show that  the composition
$\ZDer^{-1}\circ {\bf F}\circ  {\bf \Phi^{-1}}\circ \ZDer$
 satisfies the relations \eqref{m=1}--\eqref{eq:red2}  characterizing
 the  bijection $\Psi$. The proof is based on 
 Lemmas \ref{F}, \ref{fi}
and \ref{phiR2-3},  which  will be proved  in Section~\ref{The proof of three lemmas}.

Let $\tau=\langle\sigma, i_1,i_2,\ldots, i_m\rangle\in \S_{n+m}$, where $\sigma=x_1\ldots x_n\in \D_n$
 is the derangement part of $\tau$. 
Since the positions of fixed points of $\tau$ are the same as 
that of zeros of $\ZDer \tau$, we write 
 $$
 \ZDer \tau =[\sigma, i_1,i_2,\cdots,i_m]
  \in \Sh(0^m\sigma).
 $$
Thus the $k$-th zero from left to right  of  $ \ZDer \tau$
is at the position $i_k+k$  for $k\le m$ and
$i_k$ is  the  number of pillars at the left of the $k$-th zeros  in  $ \ZDer \tau$.
Consequently,  writing  $i:=i_m$ then
\begin{equation}\label{pilier}
  [x_1 x_2 \cdots x_n,i_1,i_2,\cdots,i_m]=  w x_i 0^r x_{i+1} x_{i+2}\cdots x_{n},
\end{equation}
where $w$ is a word with $i-1$ pillars and  $r$  the largest integer
satisfying  $i_{m+1-r}=i$.
The integer  $r$ must be positive  for  the last  zero
is located just at the left of the   $(i+1)$-th  pillar.
\medskip

\example
Let
$\Omega=[3\,1\,4\,5\,2,0,0,0,1,1,2,2,2,{\bf 2}]$. Then
$i=2$, $r=4$, $w=0\, 0\, 0\, 3\, 0\, 0\,$ and
$$
\Omega=w \,1 \,0^4\,4\, 5\, 2
=0\, 0\, 0\, 3\, 0\, 0\, \textbf{1}\, 0\, 0\, 0\, 0\, 4\, 5\, 2.
$$

For brevity we introduce the following notations:   $\forall t\ge0$ and 
$I= (i_1,\cdots,i_m)\in \N^m$, let $I+t=(i_1+t,\cdots,i_m+t)$, 
$t^{[m]}=(t,\cdots,t)\in \textbf{N}^m$.  Moreover,
for any $\omega\in \Sh(0^{n-m}v)$, if 
  $\F(\omega)=[\sigma,i_1,\cdots,i_m]$,  then we write
$$
 \F(\omega)=[\sigma,I_{\omega}]\quad \text{with}\quad  I_{\omega}=(i_1,\cdots,i_m).
$$
Additionally, for any non empty finite word $w$ we denote, respectively,
  by $\First(w)$ and  $\Last(w)$ the first and last letter  of $w$ when $w$ is read from left to right.
\begin{lem}\label{F}  Let
 $w_1$,  $w_2$  be two non empty words
such that  $\Last(w_1)>0$ and $\zero(w_2)=0$.  Assume that ($\Last(w_1),\,\First(w_2))=(a,b)$
	and $\pil(w_1)=\nu$, $\des(w_2)=t$.
Let $\mu=w_10^r w_2$ with $r\ge0$.
Then
 \begin{align}\label{3.1}
  \zero(w_1)\not=0\Longrightarrow    I_{\mu}=    \begin{cases}
       (t^{[r]},\,I_{w_1}+t+1),  &  \textrm{if $a>b$},\\
         I_{w_1}+t ,&  \textrm{if  $a<b$ and $r=0$},\\
           (t^{[r-1]},\,I_{w_1}+t+1,\,\nu+t), &\textrm{if $a<b$ and $r>0$};
      \end{cases}
    \end{align}
and 
    \begin{align}\label{3.2}
\zero(w_1)=0\Longrightarrow   I_{\mu}=    \begin{cases}
       \emptyset,  &  \textrm{ if  $r=0$},\\
         (t^{[r]}) ,&  \textrm{if   $r>0$ and $a>b$},\\
        (t^{[r-1]},\nu+t), &\textrm{if  $r>0$ and $a<b$}.
      \end{cases}
    \end{align}
 \end{lem}
 Let $\sigma=x_1x_2\cdots x_n\in S_n$.
The color of
 a \emph{slot}  $(i,i+1)$ of $\sigma$  can be  characterized (see \cite{chz97}) as follows:
\begin{itemize}
\item The  slot  $(i,i+1)$ of  $\sigma$ is green  if and only if
one of the following conditions is satisfied:
\begin{align}\label{eq:G}
 (G_1)\quad  x_i > x_{i+1}>i;\qquad  (G_2)\quad  x_i<i<x_{i+1};\qquad (G_3)\quad i>x_i > x_{i+1}.
\end{align}
\item The slot  $(i,i+1)$ of $\sigma$   is  red \, if and only if
one of the following conditions is satisfied:
\begin{align}\label{eq:R}
 (R_1)\quad  i<x_i <x_{i+1};\qquad  (R_2)\quad  x_i>i\ge x_{i+1};\qquad  (R_3)\quad x_i < x_{i+1}\le i.
\end{align}
\end{itemize}
By convention   $x_0=x_{n+1}=+\infty$.
If we denote by $d_i$  the  number of  descents of $\langle\sigma,i\rangle$ at right of
 $i$ ($\sigma\in \mathcal{D}_n$) then Lemma~\ref{lem clark} implies that  $g_i=d_i$ if the
 $i$-th   slot is green,  and
  $g_i=d_i+i$ if the slot is red.

\begin{lem}\label{t=di}
 Let $\sigma=x_1x_2\cdots x_n\in S_n$, let $(i,i+1)$ be the $j$-th slot of $\sigma$ and
$t_i=\des(x_{i+1}\cdots x_n)$  ($1\leq i\leq n$)  then
$$
d_j=\left\{\begin{array}{ll}
t_i&\text{ in the cases \,} \, G_1, \, G_2,\,  \text{or} \, R_1; \\
t_i+1&\text{ in the cases\, } \, G_3, \, R_2,\,  \text{or} \, R_3.
          \end{array}\right.
$$
\end{lem}

\begin{proof}
Clearly, in the cases  $G_1$, $G_2$,  $R_1$ the value  $x_{i+1}$ is  an exceedant of
 $\sigma$,  so $i+1$ is not a descent  place of $\langle\sigma,j\rangle$
 while in the cases  $G_3$, $R_2$  and $R_3$ the value $x_{i+1}$ is  a sub-exceedant,
 so $i+1$ is a descent place of $\langle\sigma,j\rangle$.
\end{proof}

Let   $\tau=\langle\sigma,i_1,i_2,\cdots,i_m\rangle \in S_{n+m}$ and
$\Omega=\ZDer\tau=[\sigma,i_1,\cdots,i_m]$.
We distinguish two cases according  to the color of the insertion slot.

\subsection{The $i_m$-{th} slot of $\sigma$ is green}
Define
$$\omega=[\sigma, i_1,\ldots, i_{m-1}],\quad
 \omega'= {\bf \Phi}^{-1}(\omega)\quad\text{and} \quad
 \Omega'={\bf \Phi}^{-1}(\Omega).
 $$
We must check (\ref{m=1}) and for $m>1$,
 \begin{align}\label{V}
 {\bf F} \circ{\bf \Phi^{-1}}(\omega)=[\sigma,j_2,\cdots,j_{m}]\Longrightarrow
{\bf F}\circ {\bf \Phi^{-1}}(\Omega)=[\sigma,g_{i_m},j_2,\cdots,j_{m}],
\end{align}
which  is equivalent to
\begin{equation}\label{IV}
I_{\Omega'}=(g_{i_m}, I_{\omega'}).
\end{equation}
\begin{lem}\label{fi}
We have the following factorizations:
$$
\omega'=w_10^rw_20^{r'}w_3 \quad\text{ and}
\quad \Omega'=w_10^{r+1}w_20^{r'}w_3
\quad (r,\,r'\ge0),
$$
where  $w_1\not=\emptyset$,  $w_2\not=\emptyset$ and  $w_3$  are  words on non negative integers. 
Moreover, if $(a,b,a',b')=(\Last(w_1),\,\First(w_2),\, \Last(w_2),\, \First(w_3))$, then
the following properties hold true:
   \begin{itemize}
  \item[ i)] $\zero(w_2w_3)=0$, and  $\Last (w_1)>0$,
  \item[ii)] if $r=0$,   then $ a>b$,
  \item[ iii)] if $r'>0$, then $r'=1$,  $a'<b'$ and  $r=0$,
  \item[ iv)] $\des(w_20^{r'}w_3)=g_{i_m}$.
\end{itemize}
\end{lem}
  Notice that  if $r'=1$ then $w_3\not=\emptyset$ by iii).
Let $w'_2=w_2w_3$.
If $m=1$  then  $\omega=\omega'=\sigma$ and $\Omega=[\sigma,i_1]$.
Hence  $\zero(\omega)=0$ and $r=r'=0$.
Thus $\omega'=w_1w'_2$ and $\Omega'=w_10w'_2$,
where $a>b=\First (w'_2)$ and
 $\des(w'_2)=g_{i_1}$. It  follows then from
   \eqref{3.2} that
\begin{equation}\label{m=1 vert}
 I_{\Omega_1'}=(g_{i_1}).
\end{equation}
We now prove  (\ref{IV})  for  $m>1$.
\begin{itemize}
\item[ 1)]  If $r'=0$  then
  \begin{equation}\label{r'=0 omega'-Omega'}
    \omega'=w_10^rw'_2 \qquad \text{and }\qquad \Omega'=w_10^{r+1}w'_2,
  \end{equation}
where   $\des(w'_2)=t=g_{i_m}$ and
 $\zero(w'_2)=0$. As
 $m>1$ we cannot have
   $\zero(w_1)=0$ and $r=0$ simultaneously. It remains to verify  the following two cases:
   \begin{itemize}
\item[i)] $\zero(w_1)\not=0$,   $a>b$ or $r>0$ and $a<b$,
\item[ii)] $\zero(w_1)=0$, $r>0$ and  $a\neq b$.
\end{itemize}
\noindent Applying Lemma \ref{F} to either case yields
 \begin{equation}\label{I Omega'-omega'}
   I_{\Omega'}=(t,I_{\omega'})=(g_{i_m},I_{\omega'}).
\end{equation}
\item[2)] If $r'=1$,  then  $a'<b'$, $r=0$ and $a>b$. Hence
$$
 \omega'=v0w_3\qquad \text{and}\qquad \Omega'=V0w_3,
$$
where $v=w_1w_2$ and $ V=w_10w_2$.
Let  $t_2=\des(w_2)$,\,$t_3=\des(w_3)$ and  $t=\des(w_20w_3)=g_{i_m}$, $\nu=\pil(v)=\pil(V)$.
If $\zero(w_1)\not=0$,  as  $a>b$,  by  \eqref{3.1} we have 
\begin{equation}\label{pi w_1}
 I_{v}=I_{w_1}+t_2+1 \quad\text{and}\quad I_{V}=(t_2,I_{w_1}+t_2+1)=(t_2,I_{v}).
 \end{equation}
Since  $\zero(v)>0$, $\zero(V)>0$ and  $a'<b'$,  by  \eqref{3.1} we have 
\begin{equation}\label{omega' pi}
  I_{\omega' }=(I_{v}+t_3+1,\nu+t_3) \quad\text{and}\quad I_{\Omega' }=(I_{V}+t_3+1,\nu+t_3).
 \end{equation}
It follows that $I_{\Omega' }=(t_2+t_3+1,I_{\omega'})$
and (as $a'<b'$)  $t=t_2+t_3+1$.
Now, if $\zero(w_1)=0$ then, by  \eqref{3.2},
\begin{equation}\label{Pi=t}
I_{\omega' }=(\nu+t_3) \qquad\text{and}\qquad I_{V}=(t_2).
 \end{equation}
As  $\zero(V)>0$,  it follows from   (\ref{omega' pi})   and   (\ref{Pi=t})  that
$ I_{\Omega' }=(t,\nu+t_3)=(t,I_{\omega' })$.
 \end{itemize}
\subsection{The $i_m$-th slot of  $\sigma$ is red}
Let $k$ be the largest integer such that
$i_{m-k+1}=i_m$. Define
  $$
  \omega=[\sigma,i_1,\cdots,i_{m-k}], \quad
  \omega'= {\bf \Phi^{-1}}(\omega),  \quad \Omega'= {\bf \Phi^{-1}}(\Omega)\quad\text{and}\quad
  \widetilde{\Omega}=\psi_{m-k}\cdots\psi_1(\Omega).
  $$
For  $k=m$  the relation  (\ref{eq:red1})  is equivalent to
\begin{equation}\label{eq:red1bis}
I_{\Omega'}=((g_{i_m}-i_m)^{[m-1]},g_{i_m}),
\end{equation}
which  corresponds to   (\ref{m=1}) when $m=1$.
For $k<m$  the relation  (\ref{eq:red2})   is equivalent to
\begin{equation}\label{eq:red2bis}
I_{\Omega'}=((g_{i_m}-i_m)^{[k-1]},I_{\omega'}+1,g_{i_m}).
\end{equation}
 We now verify (\ref{eq:red1bis}) and
  (\ref{eq:red2bis}).
  Recall that  $\sigma=x_1\ldots x_n$.  For convenience,  we use write $i:=i_m$ in what follows.  
  
  \subsubsection{ \bf Assume that $x_{i+1}$ is excedent.}
This corresponds to the case  ($R_1$), i.e.,
  $x_i$ and $x_{i+1}$ are both excedent and  $x_i<x_{i+1}$. Hence,  all the zeros of $\omega$ are at the left of
   $x_i$ and    all the zeros of  $\omega'$  remains at the left of $x_i$.  It follows that
\begin{equation}\label{R1}
    \omega'=w_1x_ix_{i+1}\cdots x_n
    \quad \text{and} \quad \widetilde{\Omega}=w_1x_i0^kx_{i+1}\cdots x_n
\end{equation}
Besides, as the map $\psi_j$ is identity for
 $m-k+1\le j\le m$,  we have  $\Omega'= \widetilde{\Omega}$.
Let  $t=\des(x_{i+1}\cdots x_n)$. By Lemma \ref{t=di}  we have   $t=g_{i}-i$.\\
If $\zero(w_1)=0$ then
 $m=k$. From \eqref{3.2} we derive 
  $ I_{\Omega'}=(t^{[k-1]},i+t)=((g_i-i)^{[m-1]},g_i) $, which is  (\ref{eq:red1bis}).
If $\zero(w_1)\not=0$, by \eqref{3.1} we have
$$
I_{\omega'}=(I_{w_1x_i}+t) \qquad \text{and}\qquad I_{\Omega'}=(t^{[k-1]},I_{w_1x_i}+t+1,i+t),
 $$
which is  precisely (\ref{eq:red2bis}).

\subsubsection{\bf  Assume that $x_{i+1}$ is subexcedent.}
 We need the following result.
\begin{lem}\label{phiR2-3}
Let $h$ be the largest integer such that
$i+1>x_{i+1}>\cdots>x_h$ and  $\ell=i+1$ if $x_i$ is excedent or the last zero of $w'$ is located between $x_i$ and $x_{i+1}$   otherwise $\ell$  be  the smallest integer such that
 $x_{\ell}<\cdots< x_{i+1}<i+1$ and that
 $\omega'$ does not contain \emph{zero} at the right of $x_{\ell}$.
Set $T=\des(x_\ell \cdots x_n)$; $w_3=x_{h+1}\cdots x_n$; $t'=\des(w_3)$,
then we have
  \begin{align}\label{omega red}
\omega'&=w_1x_{\ell-1}0^rx_\ell\cdots x_hw_3,\\
\Omega'&=w_1x_{\ell-1}0^{r+k-1}x_\ell\cdots x_h0w_3.\label{Omega red}
  \end{align}
Moreover, the following identities hold true:
\begin{itemize}
\item [\emph{i})]  $(r=0\;\text{and} \;x_{\ell-1}>x_\ell)$  or $(r=1\quad  \text{and} \quad
x_{\ell-1}<x_\ell)$,
\item [\emph{ii})]  $T+1=g_{i}-i$,
\item [\emph{iii})]  $h+t'=g_{i}$.
\end{itemize}
\end{lem}
\medskip
Let   $t=\des(x_\ell\cdots x_h)$. As $x_h<x_{h+1}$  we have $T=t+t'$.  Write  $\Omega'=V0w_3$ with
 $$
 V=w_1x_{\ell-1}0^{r+k-1}x_\ell\cdots x_h.
 $$
 \begin{itemize}
\item Suppose  that $k=m$. Then
$r=0$,  $\zero(w_1x_{\ell-1})=0$ and
 $x_{\ell-1}>x_\ell$.
\begin{itemize}
  \item If $k=1$ then $\zero(V)=0$. By \eqref{3.2} we have
 $I_{\Omega'}=(h+t')=(g_i)$, which is equivalent to
  (\ref{m=1}).
  \item If $k>1$ then $\zero(V)>0$.  By  \eqref{3.1} we have
  $I_{\Omega'}=(I_{V}+t'+1,h+t')$, while
 \eqref{3.2} yields
   $I_{V}=t^{[k-1]}$. Hence
$I_{\Omega'}=((T+1)^{[k-1]},h+t')$, which is equivalent to  (\ref{eq:red1bis}).
\end{itemize}
\item
Suppose that  $k<m$.
There are three cases:
\begin{itemize}
\item[(a)] $r=0$,   we have $\zero(w_1)\not=0$,
$x_{\ell-1}>x_{\ell}$ and, by \eqref{3.1},
  $$ I_{\omega'}=(I_{w_1x_{\ell-1}}+T+1)\qquad\text{and}\qquad I_{V}=(t^{[k-1]},
   I_{w_1x_{\ell-1}}+t+1).
   $$
\item[(b)] $r=1$ and $\zero(w_1)\not=0$, we have
 $x_{\ell-1}<x_{\ell}$. By \eqref{3.1}
  $$
  I_{\omega'}=(I_{w_1x_{\ell-1}}+T+1,\ell-1+T)\quad\text{and}\quad I_{V}
  =(t^{[k-1]}, I_{w_1x_{\ell-1}}+t+1,\ell-1+t).
  $$
\item[(c)] $r=1$ and $\zero(w_1)=0$, by \eqref{3.2} we have
  $$
   I_{\omega'}=(\ell-1+T)\qquad\text{and}\qquad I_{V}=(t^{[k-1]},\ell-1+t).
  $$
\end{itemize}
  On the other hand,
 in any case,  we have $\zero(V)>0$.  By \eqref{3.1},
  $$
  I_{\Omega'}=(I_{V}+t'+1,h+t')=((T+1)^{[k-1]},I_{\omega'}+1,h+t'),
  $$
  which is equivalent to  \eqref{eq:red2bis}.
\end{itemize}

\example Let  $\tau=1\,5\,3\,4\,2\,7\,6\,8\in S_8$ then  $\tau=\langle\sigma,0,1,1,4\rangle$,
where $\Der\tau=\sigma=2\,1\,4\,3$.
 So $\langle\sigma,0\rangle=1\,3\,2\,5\,4$; $\langle\sigma,1\rangle=3\,2\,1\,5\,4$; $\langle\sigma,2\rangle=2\,1\,3\,5\,4$;
 $\langle\sigma,3\rangle=2\,1\,5\,4\,3$. $\langle\sigma,4\rangle=2\,1\,4\,3\,5$; hence the slots 0, 2, 4 are green, while the slots 1 and 3 are red. Therefore $(g_0,\ldots ,g_4)=(2,3,1,4,0)$, and
 \begin{itemize}
\item slot 0 is green $\Longrightarrow \Psi\langle\sigma,0\rangle=\langle\sigma,g_0\rangle=\langle\sigma,2\rangle$,
\item slot 1 is red $\Longrightarrow\Psi\langle\sigma,0,1,1\rangle=\langle\sigma,g_1-1,2+1,g_1\rangle=\langle\sigma,2,3,3\rangle$,
\item slot 4 is green $\Longrightarrow\Psi\langle\sigma,0,1,1,4\rangle=\langle\sigma,g_4,2,3,3\rangle=\langle\sigma,0,2,3,3\rangle$.
\end{itemize}
Thus $\Psi(\tau)=\langle\sigma,0,2,3,3\rangle=1\,3\,2\,4\,8\,6\,7\,5$.

On the other hand, applying  $\Phi^{-1}$ to  $w:=\ZDer\tau=02001430\in \Sh(0^4\sigma)$, as 2 and 4 are excedent, we get
$$
\begin{array}{rcccccccccll}
  Id && 1 & 2 & 3 & 4 & 5 & 6 & 7 & 8 \\
  w&= & 0 & 2 & 0 & 0 & 1 & 4 & 3 &0 \\
  \psi_1(w)&=&0 & 2 & 0 & 0 & 1 & 4 & 3 &0&\text{Case (1')}\\
  \psi_2( \psi_1(w))&=&0 & 2 & 0 & 0 & 1 & 4 & 3 &0&\text{Case (1')}\\\
  \psi_3( \psi_2( \psi_1(w)))&=&0 & 2 & 0 & 1 & 0 & 4 & 3 &0&\text{Case (3') with $k=5$}\\
  \psi_4( \psi_3( \psi_2( \psi_1(w))))&=&0 & 2 & 0 & 1 & 0 & 4 & 0 &3&\text{Case (2') with $k=7$}.\\
\end{array}
$$
Thus ${\bf\Phi}^{-1}(w)=02010403$. Now we apply $\F$ to  $w=02010403$:
\begin{align*}
\F(02)&=02  &\text{no descent}\\
\F(020)&=\delta(02)0=200 &\text{Case (3)}\\
\F(0201)&=\gamma(200)1=0201  &\text{Case (2)}\\
\F(02010)&=\delta(0201)0=20100  &\text{Case (3)}\\
\F(020104)&=201004\quad  &\text{Case (1)}\\
\F(0201040)&=\delta(020104)0=2104000 &\text{Case (3)}\\
\F(02010403)&=\gamma(2104000)3=02104003 &\text{Case (2)}.
\end{align*}
So  $\F(02010403)=[\sigma,0,2,3,3]$ and  $\ZDer^{-1}([\sigma,0,2,3,3])=\langle\sigma,0,2,3,3\rangle=1\,3\,2\,4\,8\,6\,7\,5$.
Therefore \eqref{eq:FHdecom} is checked.


\section{The proof of three lemmas}\label{The proof of three lemmas}
\subsection{Proof of Lemma \ref{F}}
Recall that
$\delta$ and  $\Upsilon$ are the  transformations   used in the cases
 (b) and (c) of the  algorithm ${\bf F}$.
 Let $\alpha$ a word of length   $\nu$ on the alphabet of positive integers and
  $b>0$.   Let $w=[\alpha,\,i_1,\cdots, i_z]\in \Sh(0^z\alpha)$.
Note that
 $$
 w b=[\alpha b,\,i_1,\cdots, i_z],\quad
  w0=[\alpha,\,i_1,\cdots i_z,\,\nu],\quad 0w=[\alpha,0,\,i_1,\cdots, i_z].
  $$
Hence, if
 $\Last\,(w)=0$, i.e., $i_z=\nu$, then
$\Upsilon[\alpha,\,i_1,\cdots, i_z]=[\alpha,0,\,i_1,\cdots, i_{z-1}]$.
On the other hand,  if
$x_1,\cdots x_n$ are positive integers, then
 $$x_1\cdots x_{k}\,0^r\,x_{k+1}\cdots x_n=[x_1\cdots
  x_n,k^{[r]}].
  $$
It follows that $\delta(w)= [\alpha,\,i_1+1,\cdots, i_z+1]$. Now, consider the word
 $\mu=w_10^rw_2$, where $\Pil(\mu)=\sigma$, $\Pil(w_1)=\alpha$ and
  $|\alpha|=\nu$.
   Since  $w_1$ and $w_2$ are non empty we can write
   $w_1=v_1a$ and $w_2=bv_2$.  Therefore  $\mu=v_1a0^rbv_2$.

 If $\zero(w_1)=z\not=0$ and $a>b$  set
$I_{w_1}=(\ell_1,\cdots,\ell_z)$. Then, we have successively
$$
\F(v_1a0)=\delta(\F(v_1a))0=[\alpha,\ell_1+1,\cdots,\ell_z+1,\nu] \Longrightarrow I_{v_1a0}=(I_{v_1a}+1,\nu),
$$
$$\F(v_1a00)=\Upsilon(\F(v_1a0))0=[\alpha,0,\ell_1+1,\cdots,\ell_z+1,\nu] \Longrightarrow I_{v_1a0^r}=(0^{[r-1]},I_{v_1a}+1,\nu),
 $$
 and  $\F(v_1a0^rb)=\Upsilon(\F(v_1a0^{[r]}))b=[\alpha b,0^r,I_{w_1}+1]$.
Finally, as $\zero(v_2)=0$, we have
$\F(\mu)=[\sigma,t^{[r]},I_{w_1}+1+t]$, which corresponds to the first case of \eqref{3.1}.
 The other cases can be proved similarly.
\subsection{Proof of Lemma \ref{fi}}
Recall that $\Omega=[\sigma,\,i_1,\cdots,i_m]$, where
$$
\sigma=x_0x_1\cdots x_nx_{n+1}, \quad
\omega=[\sigma,\,i_1,\cdots,i_{m-1}], \quad
\Omega'=\f(\Omega), \quad\omega'=\f(\omega).
$$
Set $\widetilde{\omega}=\psi_{m-2}\circ\cdots\circ\psi_1(\omega)$,  $\widetilde{\Omega}=\psi_{m-1}\circ\cdots\circ\psi_1(\Omega)$ and $i=i_m$.
 The last zero of $\omega$ is at the left of
$x_{i+1}$ so is the last zero of
 $\widetilde{\omega}$ by definition of  $\f$.
So all $m-2$ zeros of $\omega'$ are on the left of $x_{i+1}$ and only the last can be on the left or on the right of $x_{i+1}$. Hence  $\omega'$ is of the following form:
$$
\omega'=v_1x_i0^\epsilon x_{i+1}x_{n+2}\ldots x_h0^{r'}x_{h+1}\ldots x_nx_{n+1},
 $$
where $\epsilon\ge0$, and  $0\le r'\le 1$.
Keeping in mind the definition of
$\f$, if $x_{i+1}$  is excedent (the case of $G_1$ or $G_2$) then we have  necessarily $r'=0$ and If $x_{i+1}$  is subexcedent   (the case of $G_3$) then
 it is impossible for the last zero of  $\omega'$ to be between $x_i$ and $x_{i+1}$, otherwise when one applies to $\omega'$ the map $\psi_{m-1}^{-1}$, it corresponds to the case (2) therefore the last \emph{zero} of $\widetilde{\omega} $ would be
   on the right of $x_{i+1} $,  that is impossible. So in case of $G_3$ we have necessarily $r'=1$ or ($r'=0$ and $ \epsilon=0$).    We consider the following three  cases:
  \begin{itemize}
\item   $G_1$ or ($G_2$ and $\epsilon>0$). Since  $r'=0$, hence $\omega'$ and  $\widetilde{\omega} $ are of the following forms:
\begin{align*}
\omega'=v_1x_i0^\epsilon x_{i+1}\ldots x_{n+1},\qquad
\widetilde{\Omega}=v_1x_i0^{\epsilon+1} x_{i+1}\ldots x_{n+1}.
\end{align*}
In this case $\psi_m$ is the identity  therefore $\Omega'=\widetilde{\Omega}$ and then  we have
$$
r=\epsilon; \,r'=0\quad w_1=v_1x_i;\quad w_2=x_{i+1}\cdots x_nx_{n+1};\quad w_3= \emptyset.
$$
If $r=0$ we have $G_1$ so  $a>b$ and by Lemma \ref{t=di},
$\des(w_2)=g_i$.
 \item ($G_2$ and $\epsilon=0$) or ($G_3$ and $r'=0$).
 In this case, let  $k$ be the smallest integer such that $x_k<\cdots<x_i<i$  and  $\omega'$ does not
 contain any zero on the right  of $x_k$. Thus
 $\omega'$ and $\widetilde{\Omega}$ are of the following forms:
     \begin{align*}
     \omega'&=v_1x_{k-1}0^rx_{k}\ldots x_i x_{i+1}\ldots x_{n+1},\\
     \qquad \widetilde{\Omega}&=v_1x_{k-1}0^r x_{k}\ldots x_i0x_{i+1}\ldots x_{n+1}.
     \end{align*}
Since $\Omega'=\psi_m(\widetilde{\Omega})$ and the map $\psi_m$ corresponds to the case (2'),
we have
 $$\Omega'=v_1x_{k-1}0^{r+1} x_{k}\cdots x_i x_{i+1}\cdots x_nx_{n+1}.
 $$
Moreover $r'=0\quad w_1=v_1x_{k-1}$ and $w_2=x_{k}\cdots x_nx_{n+1};\quad w_3= \emptyset$.
As  $\des(x_{k}\cdots x_i)=0$,  we have $\des(w_2)=g_i$ by Lemma \ref{t=di}.
\item  ($G_3$ and $r'=1$). In this case, both $x_i$ and $x_{i+1}$ must be  subexcedances and $x_i>x_{i+1}$ and
  $h$ must  be the largest integer such that  $i+1>x_{i+1}>\cdots>x_h$ (by applying $\phi_{m-1}$ to  $\widetilde{\omega}$).
Let $k$ be the smallest integer such that  $k>x_k>\cdots >x_{i+1}$ and  $\omega'$  contains only one zero on the right   of $x_k$.  Thus we have
$$
\omega'=v_1x_{k-1} 0^\alpha x_{k}\cdots x_i x_{i+1}\cdots x_h0x_{h+1}\cdots x_{n+1}.
$$
On the other hand, the last zero of $\widetilde{\omega}$ is at the left  of $x_{i+1}$,
more precisely just at the left  of $x_k$
because $k>x_k>x_i>x_{i+1}>\cdots>x_h$ by the map $\phi_{m-1}$.
Hence
$$
\widetilde{\omega}=v_1 x_{k-1} 0^{\alpha+1}x_{k}\cdots x_i x_{i+1}\cdots x_hx_{h+1}\cdots x_{n+1}.
$$
Set $\widetilde{\widetilde{\Omega}}=\psi_{m-2}\circ\cdots \circ \psi_1(\Omega)$.
Then
$$
\widetilde{\widetilde{\Omega}}=[\widetilde{\omega},i]=v_1x_{k-1} 0^{\alpha+1}x_{k}\cdots x_i0 x_{i+1}\cdots x_hx_{h+1}\cdots x_{n+1}.
   $$
 When we apply $\psi_{m-1}$ to  $\widetilde{\widetilde{\Omega}}$,
it corresponds to the  case (3'), so
 $$
\widetilde{\Omega}=\psi_{m-1}(\widetilde{\widetilde{\Omega}})=v_1 x_{k-1} 0^{\alpha}x_{k}\cdots x_i00x_{i+1}\cdots x_hx_{h+1}\cdots x_{n+1}.
 $$
Similarly we have
  $$\Omega'=\psi_m(\widetilde{\Omega})=v x_{k-1} 0^{\alpha}x_{k}\cdots x_i0x_{i+1}\cdots x_h0x_{h+1}\cdots x_{n+1}.
   $$
In this situation we take $r=0$ ($r'=1$) and
 $$w_1=v_1x_{k-1} 0^\alpha x_{k}\cdots x_i;\quad w_2= x_{i+1}\cdots x_h;\quad w_3= x_{h+1}\cdots x_{n+1}.
$$
As $x_h<x_{h+1}$ we have   $\des(w_20w_3)=\des(w_2w_3)+1=g_i$ by Lemma~\ref{t=di}.
\end{itemize}

\subsection{Proof of Lemma \ref{phiR2-3}}
Since the  $i$-th slot of $\sigma$ is red and  $x_{i+1}$ is  subexcedent, we are in the situation of  $R_2$ or $R_3$. Recall that all the \emph{zeros} of $\omega$  are located at the left of
  $x_i$ so are all  of $\widetilde{\omega}=\psi_{m-k-1}\circ\cdots\circ\psi_1(\omega)$.
 We show that the last zero of  $\omega'$ is at the left of  $x_{i+1}$.
In the $R_2$ case $x_i$ is excedent and all zeros of $\omega'$ are on the
left of $x_i$. In the $R_3$ case, suppose that the last zero of  $\omega'$ is at the right of
$x_{i+1}$,  then, by applying the reverse mapping
$\psi_{m-k}^{-1}=\phi_{m-k}$, the last
\emph{zero} of $\widetilde{\omega}$ cannot be at the left of
$x_i$ because
 $x_i<x_{i+1}$. This  is absurd. In order to show   (\ref{omega red}) and   (\ref{Omega red})
 set $\Omega_{m-k}=\widetilde{\Omega}$ and for all $j$ such that $m-k<j\le m$ set $\Omega_j=\psi_j(\Omega_{j-1})$.
There are  two cases:
 \begin{itemize}
  \item  If $\ell\le i$ then
$$\omega'= w_1x_{\ell-1}0^rx_\ell\cdots x_{i}x_{i+1} \cdots x_hx_{h+1}w_3,
$$ and
$$
 \widetilde{\Omega}= w_1x_{\ell-1}0^rx_\ell\cdots x_{i} 0^kx_{i+1} \cdots x_hx_{h+1}w_3.
 $$
Noticing  that  $x_i<x_{i+1}<i+1$, so
$x_i$ is subexcedent and the application of $\psi_j$  to   $\Omega_{j-1}$
  corresponds to case  (2')  for all   $m-k<j<m$ and
   $\psi_m$ corresponds to case  (3'). Therefore
$\Omega'= w_1x_{\ell-1}0^{r+k-1}x_\ell\cdots x_ix_{i+1} \cdots   x_h0x_{h+1}w_3$.
  \item  If $\ell=i+1$, then
$\omega'= w_1x_{i}0^rx_{i+1} \cdots x_hx_{h+1}w_3$ and
$\widetilde{\Omega}= w_1x_{i}0^{r+k}x_{i+1} \cdots x_hx_{h+1}w_3$.
 In this case,  $\psi_j$   corresponds to the case
    (1) for all  $m-k<j<m$,  and  $\psi_m$ corresponds to case (3').  So
 $\Omega'= w_1x_{i}0^{r+k-1}x_{i+1} \cdots  x_h0x_{h+1}w_3$.

\end{itemize}
It remains to verify the three
conditions of Lemma \ref{phiR2-3} in the above two cases.
It is clear that $x_{\ell-1}> x_{\ell}$ if $r=0$. Moreover, neither
  $r>1$ nor  ($r=1$ and
 $x_{\ell-1}>x_\ell$) is possible because, otherwise,  when we apply
  $\phi_{m-k}$ to  $\omega'$,
  it corresponds to case  (2), so the last  \emph{zero} of
   $\widetilde{\omega}$
  would be  at the right of  $x_{i+1}$, but this is absurd.  So  the condition $(i)$ is verified.
Besides, as   $\des(x_\ell\cdots x_{i+1})=0$ we have $T=\des(x_{i+1}\cdots x_n)$,
  and by Lemma \ref{t=di},  we derive the  condition $(ii)$. Finally, by definition of
  $h$ we have
   $\des(x_{i+1}\cdots x_h)=h-i-1$, and  $T=\des(x_{i+1}\cdots x_n)=h-i-1+t'$. Thus $h+t'=g_i$.

\medskip
{\bf Acknowledgments}:
This work was  done during  the first author's visit to Institut Camille Jordan, Universit\'{e} Lyon 1 in the fall
 of 2008 and was
supported by  a  scholarship of Agence universitaire de la francophonie.
The second author acknowledges the financial support from la R\'egion Rh\^one-Alpes via the program MIRA recherche 2008, projet: 0803414701.
\goodbreak


\end{document}